\documentclass{amsart}
\usepackage{a4wide,amssymb}

\newcommand{\W}{\mathcal W}
\newcommand{\weak}{\mathrm{weak}}
\newcommand{\vC}{\vec{\mathcal C}}
\newcommand{\IR}{\mathbb R}
\newcommand{\IN}{\mathbb N}
\newcommand{\C}{\mathcal C}
\newcommand{\U}{\mathcal U}
\newcommand{\F}{\mathcal F}

\newtheorem{theorem}{Theorem}
\newtheorem{problem}{Problem}
\newtheorem{proposition}{Proposition}
\newtheorem{corollary}{Corollary}

\title[On topological classification of normed spaces]{On topological classification of normed spaces\\ endowed with the weak topology or\\ the topology of compact convergence}
\author{Taras Banakh}
\address{Ivan Franko National University of Lviv, Ukraine}
\email{t.o.banakh@gmail.com}

\begin{document}

\begin{abstract}
In this paper the weak topology on a normed space is studied from the viewpoint of infinite-dimensional topology. Besides the weak topology on a normed space $X$ (coinciding with the topology of uniform convergence on finite subsets of the dual space $X^*$), we consider the topology $c$ of uniform convergence on compact subsets of $X^*$. It is known that this topology coincides with the weak topology on bounded subsets of $X$, but unlike to the latter has much better topological properties (e.g., is stratifiable).

We prove that for normed spaces $X,Y$ with separable duals the spaces $(X,\weak)$, $(Y,\weak)$ are sequentially homeomorphic if and only if $\W(X)=\W(Y)$, where $\W(X)$ is the class of topological spaces homeomorphic to closed bounded subsets of $(X,\weak)$. Moreover, if $X,Y$ are Banach spaces which are isomorphic to their hyperplanes and have separale duals, then the spaces $(X,\weak)$ and $(Y,\weak)$ are sequentially homeomorphic if and only of the spaces $(X,c)$ and $(Y,c)$ are homeomorphic. To prove this result, we show that for a normed space $X$ which is isomorphic to its hyperpane and has separable dual, the space $(X,c)$ (resp. $(X,\weak)$)  is (sequentially) homeomorphic to the product $B\times\mathbb R^\infty$ of the weak unit ball $B$ of $X$ and the linear space $\mathbb R^\infty$ with countable Hamel basis and the strongest linear topology. 
\end{abstract}
\maketitle

The problem of topological classification of linear topological spaces traces its history back to M.~Fr\'echet \cite{16} and S.~Banach~\cite{3} who posed in early 30-ies the problem of topological equivalence of all infinite-dimensional separable Banach spaces. This problem was resolved affirmatively by M.I.~Kadec \cite{21} who proved that each separable Banach space is homeomorphic to a Hilbert space. Later H.~Toru\'nczyk \cite{26} extended this classification onto unseparable Banach spaces. In this context the following problem arises naturally.

\begin{problem}\label{prob1} Give topological classification of Banach spaces endowed with the weak topology.
\end{problem}

Unlike to the topology of norm, this problem turns out to be quite difficult. The reason lies in the fact that the weak topology is very far from being metrizable, see \cite{17}. However, due to this distance, it is possible to consider various kinds of topological equivalences between spaces endowed with weak topologies. Besides usual homeomorphisms between such spaces, we shall consider also sequential homeomorphisms. By a {\em sequential homeomorphism} between topological spaces $X,Y$ we understand a bijective function $h:X\to Y$  such that both $h$ and $h^{-1}$ are sequentially continuous. A function $f:X\to Y$ is {\em sequentially continuous} if $\lim_{n\to\infty}f(x_n)=f(\lim_{n\to\infty}x_n)$ for every convergent sequence $(x_n)$ in $X$. Clearly that each homeomorphism $h:X\to Y$ is a sequential homeomorphism, while the converse is true in case of sequential spaces $X,Y$. We remind that a topological space $X$ is {\em sequential} if a subset $U\subset X$ is open if and only if $U\cap K$ is open in $K$ for every compact countable subset $K$ of $X$, see \cite[\S1.6]{15}. In general, a sequential homeomorphism needs not be a homeomorphism in the usual sense, see Proposition~\ref{p1} below. So we modify Problem~\ref{prob1} to a weaker one:

\begin{problem}\label{prob2} Under which conditions are two Banach spaces endowed with the weak topology sequentially homeomorphic?
\end{problem}

We show that this is the case if the dual spaces of $X$ and $Y$ are separable and $\W(X)=\W(Y)$. Here for a locally convex spaces $X$ by $\W(X)$ we denote the class of topological spaces homeomorphic to bounded closed subspaces of $(X,\weak)$. In case of Banach spaces, the classes $\W(X)$ were introduced and studied in \cite{7} where the following classification result was proven: {\em the weak unit ball of two Banach spaces $X,Y$ with Kadec norms are homeomorphic if and only if $\W(X)=\W(Y)$}. By a {\em weak unit ball} of a Banach space we understand its closed unit ball endowed with the weak topology. A similar results holds also for sequential homeomorphisms between normed spaces endowed with the weak topology.

\begin{theorem}\label{t1} For normed spaces $X,Y$ with separable duals, the spaces $(X,\weak)$ and $(Y,\weak)$ are seqeuntially homeomorphic if and only if $\W(X)=\W(Y)$.
\end{theorem}

In fact, this theorem concerns not only the weak topology, but many other topologies coinciding with the weak topology on bounded subsets. We say that a topology $\tau$ on a normed space $X$ is {\em sequentially weak} if the identity map $(X,\tau)\to(X,\weak)$ is a sequntial homeomorphism. Evidently, each topology $\tau\supset\weak$ coinciding with the weak topology on bounded sets is sequntially weak.

\begin{corollary}\label{c1} Let $X,Y$ be normed spaces with separable duals and $\tau_X$, $\tau_Y$ be sequentially weak topologies on $X,Y$, respectively. The spaces $(X,\tau_X)$ and $(Y,\tau_Y)$ are seqeuntially homeomorphic if and only if $\W(X)=\W(Y)$.
\end{corollary}

An important example of a sequentially weak topology is the topology of uniform convergence on compact subsets of the dual space (briefly, the {\em topology of compact convergence}), see \cite[8.3.3]{14} or \cite[IV.6.3]{25}. The neighborhood base of this topology at the origin of a normed space $X$ consists of the polars $K_\circ=\{x\in X:|x^*(x)|\le 1$ for all $x^*\in K\}$ of compact sets $K$ of the dual Banach sapce $X^*$. The space $X$ endowed with the topology of compact convergence will be denoted by $(X,c)$ or $X_c$. It is interesting to notice the following

\begin{proposition}\label{p1} For an infinite-dimensional normed space $X$ with separable dual, the spaces $(X,c)$ and $(X,\weak)$ are sequentially homeomorphic but not homeomorphic.
\end{proposition}

Thus the ``sequential homeomorphness'' in Corollary 1 cannot be replaced by the usual homeomorphness. Yet, in case when $\tau_X$, $\tau_Y$ are the topologies of compact convergence, this can be done.

\begin{theorem}\label{t2} For Banach spaces $X,Y$ which are isomorphic to their hyperplanes and have separable duals, the spaces $(X,c)$ and $(Y,c)$ are homeomorphic if and only if $\W(X)=\W(Y)$.
\end{theorem}

To prove this theorem, we shall investigate the topology of the pair $(X^{**}_c,X_c)$, where $X^{**}_c$ is the second dual of a normed space $X$ endowed with {\em the dual topology of compact convergence}. A neighborhood base of this topology at the origin of $X^{**}$ consists of the polars $K^\circ=\{x^{**}\in X^{**}:|x^{**}(x^*)|\le 1$ for all $x^*\in K\}$ of compact subsets $K$ of the dual space $X^*$. This topology was studied by Arens \cite{1} and was called the {\em Arens topology} in \cite[8.3.3]{14}. According to the Banach-Dieudonn\'e Theorem \cite[IV.6.3]{25}, the dual topology of compact convergence is the strongest topology on $X^{**}$ coinciding with the $*$-weak topology on bounded subsets of $X^{**}$. Clearly, the space $X_c$ can be considered as a subspace of $X^{**}_c$.

By $\W(X^{**},X)$ we denote the class of pairs $(A,B)$ homeomorphic to a pair $(K,K\cap X)$ where $K$ is a compact subset of $X^{**}_c$, equivalently, of the second dual space $X^{**}$ endowed with the $*$-weak topology.   We recall that two pairs $(A,B)$ and $(A',B')$ are homeomorphic of $h(B)=B'$ for some homeomorphism $h:A\to A'$.

\begin{theorem}\label{t3} For normed spaces $X,Y$ with separable duals, the pairs $(X^{**}_c,X_c)$ and $(Y^{**}_c,Y_c)$ are homeomorphic if and only if $\W(X^{**},X)=\W(Y^{**},Y)$.
\end{theorem}

For a Banach space $X$ the class $\W(X^{**},X)$ was introduced and studied in \cite{7} where it was proved that this class is $[0,1]$-stable, provided the Banach space $X$ has separable dual and one of the following conditions is satisfied: (1) $X$ is isomorphic to its hyperplane; (2) $X$ is infinite-dimensional and has PCP; (3) $X$ is not strongly regular. We define a class $\vC$ of pairs to be {\em $[0,1]$-stable} if $(K\times[0,1],C\times[0,1])\in\vC$ for every pair $(K,C)\in\vC$.

At the moment we know no infinite-dimensional Banach space $X$ with separable dual for which the class $\W(X^{**},X)$ is not $[0,1]$-stable (such space, if exists, must be very exotic: it is strongly regular but fails PCP and is not isomorphic to its hyperplane).

In case of $[0,1]$-stable class $\W(X^{**},X)$ the pair $(X^{**}_c,X)$ admits a satisfactory description of its topology. Let $\IR^\infty$ denote the real linear space with countable Hamel basis, endowed with the strongest locally convex topology (such a space $\IR^\infty$ is unique up to isomorphism). A topological characterization of the space $\IR^\infty$ was given by K.~Sakai \cite{24}, see also \cite{6}. Below for a normed space $X$ by $B$ we denote the weak unit ball of $X$ and by $B^{**}$ the closed unit ball of the second dual space $X^{**}$ endowed with the $*$-weak topology. By the Goldstein Theorem \cite[Th.64]{19}, the weak unit ball $B$ is a dense subspace in the $*$-weak double dual ball $B^{**}$. Moreover, if the dual space $X^*$ is separable, then the $*$-weak ball $B^{**}$ is metrizable and compact, see Proposition 62 in \cite{19}.

\begin{theorem}\label{t4} Suppose $X$ is a normed space with separable dual. If the class $\W(X^{**},X)$ is $[0,1]$-stable, then the pair $(X^{**}_c,X_c)$ is homeomorphic to $(B^{**}\times\IR^\infty,B\times\IR^\infty)$.
\end{theorem}

Since the class $\W(X^{**},X)$ is $[0,1]$-stable for every normed space $X$ which is isomorphic to its hyperplane, Theorem~\ref{t4} implies

\begin{corollary}\label{c2} Suppose $X$ is a normed space which is isomorphic to its hyperplane and has separable dual. The space $(X,c)$ is homeomorphic to the product $B\times\IR^\infty$, where $B$ is the weak unit ball of $X$. Consequently, the space $(X,\weak)$ is sequentially homeomorphic to $B\times\IR^\infty$.
\end{corollary} 

It should be mentioned that spaces homeomorphic to products of metrizable spaces and $\IR^\infty$ have appeared in \cite{22}, \cite{20}, \cite{5}.
\smallskip

{\bf On $\vC$-injective pairs and $\C$-injective spaces.} In this section we introduce and develop our main technical tool --- $\vC$-injective pairs. By a {\em pair} we understand a pair $(X,Y)$ of topological spaces $Y\subset X$. A pair $(X,Y)$ is called {\em metrizable} if so is the space $X$.

A pair $(X,Y)$ is defined to be {\em $\vC$-injective}, where $\vC$ is a class of pairs, if $X$ carries the direct limit topology $\varinjlim X_n$ whith respect to a tower $X_1\subset X_2\subset\cdots$ of closed subspaces of $X$ such that
\begin{itemize}
\item $\bigcup_{n=1}^\infty X_n=X$ and $(X_n,X_n\cap Y)\in\vC$ for every $n\in\IN$;
\item for every $n\in\IN$ and every pair $(K,C)\in \vC$, any closed embedding $f:B\to X_n$ of a clsoed susbet $B\subset K$ with $f^{-1}(Y)=B\cap C$ can be extended to a closed embedding $\bar f:K\to X$ such that $\bar f^{-1}(Y)=C$ and $\bar f(K)\subset X_m$ for some $m\ge n$.
\end{itemize}
A topological space $X$ is called {\em $\C$-injective}, where $\C$ is a class of spaces, if the pair $(X,X)$ is $\vC$-injective for the class $\vC:=\{(C,C):C\in\C\}$.

We recall that the topology of the direct limit $\varinjlim X_n$ with respect to a tower $X_1\subset X_2\subset\cdots$ is the strongest topology on the union $X=\bigcup_{n=1}^\infty X_n$ inducing the original topology on each space $X_n$.

According to the Banach-Dieudonn\'e Theorem \cite[IV.6.3]{25}, for a normed space $X$ the dual topology of compact convergence on $X^{**}$ is the strongest topology inducing the $*$-weak topology on each bounded subset of the second dual space $X^{**}$. This means that the following statement holds.

\begin{proposition}\label{p2} For any normed space $X$, the space $X^{**}_c$ carries the direct limit topology $\varinjlim nB^{**}$ with respect to the tower $B^{**}\subset 2B^{**}\subset\cdots$, where $B^{**}$ stands for the closed unit ball of the second dual space $X^{**}$ endowed with the $*$-weak topology.
\end{proposition}

Repeating arguments of \cite{24} (see also \cite{5} and \cite{23}), one may easily prove the following uniqueness theorem.

\begin{theorem}\label{t:A}\begin{enumerate}
\item Any two $\vC$-injective pairs, where $\vC$ is a class of pairs, are homeomorphic.
\item Any two $\C$-injective spaces, where $\C$ is a class of spaces, are homeomorphic.
\end{enumerate}
\end{theorem}

Next, we prove some results supplying us with examples of $\vC$-injective pairs. First, we recall some definitions from infinite-dimensional topology, see \cite{9}. {\em All maps considered below are continuous}.

We say that two maps $f,g:X\to Y$ are {\em $\U$-close} with respect to a cover $\U$ of $Y$ if for every $x\in X$ the doubleton $\{f(x),g(x)\}$ is contained in some set $U\in\U$. A closed subset $A$ of a topological space $X$ is called a {\em $Z$-set} in $X$ if for every map $f:[0,1]^n\to X$ of a finite-dimensional cube and every open cover $\U$ of $X$ there exists a map $\tilde f:[0,1]^n\to X$ such that $\tilde f([0,1]^n)\cap A=\emptyset$ and $\tilde f$ is $\U$-close to $f$. An embedding $f:X\to Y$ is called a {\em $Z$-embedding} if $f(X)$ is a $Z$-set in $Y$.

A pair $(X,Y)$ is called {\em strongly $(K,C)$-universal}, where $(K,C)$ is a pair, if for any open cover $\U$ of $X$, any closed subset $B\subset K$, and any map $f:K\to X$ such that $f{\restriction}B$ is a $Z$-embedding with $(f{\restriction}B)^{-1}(Y)=B\cap C$, there exists a $Z$-embedding $\tilde f:K\to X$ such that $\tilde f{\restriction}B=f{\restriction}B$, $\tilde f^{-1}(Y)=C$ and $\tilde f$ is $\U$-close to $f$.

A pair $(X,Y)$ is called {\em strongly $\vC$-universal}, where $\vC$ is a class of pairs, if it is strongly $(K,C)$-universal for every pair $(K,C)\in\vC$. A pair $(X,Y)$ is called {\em strongly universal} if it is strongly $\F_0(X,Y)$ universal, where $\F_0(X,Y)$ is the class of pairs homeomorphic to the pairs $(F,F\cap Y)$ where $F$ is a closed subset of $X$. 

A space $X$ is defined to be {\em strongly $\C$-universal}, where $\C$ is a class of spaces, if the pair $(X,X)$ is strongly $(C,C)$-universal for every space $C\in\C$. A space $X$ is defined to be {\em strongly universal} if it is $\F_0(X)$-stringly universal, where $\F_0(X)$ is the class of topological spaces homeomorphic to closed subspaces of $X$.  For more detail information concerning strongly universal pairs and spaces, see \cite{11}, \cite{9}, \cite{8}. 

\begin{proposition}\label{p3} A pair $(X,Y)$ is $\vC$-injective, where $\vC$ is a class of metrizable pairs, provided $X$ carries the direct limit topology with respect a tower $X_1\subset X_2\subset\cdots$ of closed subsets of $X=\bigcup_{n=1}^\infty X_n$ such that for every $n\in\IN$ the following conditions are satisfied:
\begin{itemize}
\item[\textup{(i)}] $X_n$ is a metrizable absolute retract;
\item[\textup{(ii)}] $X_n$ is a $Z$-set in $X_{n+1}$; 
\item[\textup{(iii)}] $(X_n,X_n\cap Y)\in\vC$; 
\item[\textup{(iv)}] the pair $(X_n,X_n\cap Y)$ is strongly $\vC$-universal.
\end{itemize}
\end{proposition} 

\begin{proof} Clearly, the first condition of the definition of $\vC$-injectivity is satisfied. To verify the second one, fix $n\in\IN$, a pair $(K,C)\in\vC$, a closed subset $B\subset K$, and a closed embedding $f:B\to X_n$ with $f^{-1}(Y)=B\cap C$. By the condition (ii), $f(B)$ is a $Z$-set in $X_{n+1}$. Since the space $X_{n+1}$ is an absolute retract, the map $f$ extends to a map $g:K\to X_{n+1}$. Using the strong $\vC$-universality of the pair $(X_{n+1},X_{n+1}\cap Y)$ find a $Z$-embedding $\bar f:K\to X_{n+1}$ extending the embedding $f=g{\restriction}B$ and such that $\bar f^{-1}(Y)=C$.
\end{proof}

A particular case of Proposition~\ref{p3} is the following

\begin{proposition}\label{p4} A space $X$ is $\C$-injective, where $\C$ is a class of spaces, provided $X$ carries the direct limit topology with respect to a tower $X_1\subset X_2\subset \cdots$ of closed subspaces of $X=\bigcup_{n=1}^\infty X_n$ such that for every $n\in\IN$ the space $X_n\in\C$ is a metrizable strongly $\C$-universal absolute retract which is a $Z$-set in $X_{n+1}$.
\end{proposition}

\begin{proposition}\label{p5} If a pair $(X,Y)\in\vC$ is strongly $\vC$-universal for some $[0,1]$-stable class $\vC$ of pairs and the space $X$ is a metrizable compact absolute retract, then the pair $(X\times\IR^\infty,Y\times\IR^\infty)$ is $\vC$-injective.
\end{proposition}

\begin{proof} It follows from the topological characterization of the space $\IR^\infty$ \cite{24} that $\IR^\infty$ is homeomorphic to the direct limit $I^\infty=\varinjlim I^n$ of the tower
$$I^1\subset I^2\subset I^3\subset\cdots
$$
where $I=[0,1]$ and each $I^n$ is identified with the subset $I^n\times\{0\}$ in $I^{n+1}$.

Since the space $X$ is compact, the product $X\times\IR^\infty$ carries the direct limit topology $\varinjlim X\times I^n$ with respect to the tower $X\times I\subset X\times I^2\subset X\times I^3\subset\cdots$. To show that this tower satisfies the conditions of Proposition~\ref{p3}, fix arbitrary $n\in\IN$. Since $X$ is an absolute retract, so is the product $X\times I^n$. Because the class $\vC$ is $[0,1]$-stable, $(X\times I^n,Y\times I^n)\in\vC$. Evidently, $I^n=I^n\times\{0\}$ is a $Z$-set in $I^{n+1}$. This implies that $X\times I^n$ is a $Z$-set in $X\times I^{n+1}$. Finally, the strongly $\vC$-university of the pair $(X,Y)$ implies the strongly $\vC$-university of the pair $(X\times I^n,Y\times I^n)$, see \cite[10.5]{2}, \cite[4.4]{13} or \cite[1.13]{8}.
\end{proof}

Next, we use some facts concerning the strong universality in convex sets. A subset $A$ of a linear space is called {\em symmetric} if $A=-A$. A class $\C$ of spaces is defined to be {\em local\/} if a separable metrizable space $X$ belongs to the class $\C$ provided each point of $X$ has a neighborhood belonging to the class $\C$.

In the following theorem we unify Main Theorem of \cite{4}, Proposition 1.5 of \cite{8}, and Proposition 4.3 of \cite[p.158]{10}.

\begin{theorem}\label{t:B} Suppose $X$ is a linear subspace of a locally convex linear topological space $L$ and $C$ is a closed convex symmetric infinite-dimensional subset in $X$ such that the closure $\bar C$ of $C$ in $L$ is compact and metrizable. Then 
\begin{itemize}
\item[\textup{(i)}] the pair $(\bar C,C)$ and the space $C$ are strongly universal;
\item[\textup{(ii)}] the class $\F_0(C)$ is local provided $C\ne\bar C$;
\item[\textup{(iii)}] $t\cdot C$ is a $Z$-set in $C$ for every $t\in[0,1)$.
\end{itemize}
\end{theorem}

\begin{proposition}\label{p6} For every normed space $X$ with separable dual the pair $(X^{**}_c,X_c)$ is $\W(X^{**},X)$-injective.
\end{proposition}

\begin{proof} By Proposition 2, the space $X^{**}_c$ carries the topology of the direct limit $\varinjlim nB^{**}$ of the tower $B^{**}\subset 2B^{**}\subset\cdots$, where $B^{**}$ stands for the closed unit ball of the second dual space $X^{**}$ endowed with the $*$-weak topology. It is well-known that the space $B^{**}$ is compact and metrizable (since the dual Banach space $X^{*}$ is separable). Consequently, the space $B^{**}$, being a metrizable convex set in a locally convex space, is an absolute retract according to the Dugundji Theorem \cite[II.\S3]{10}. It follows from the definition of the class $\W(X^{**},X)$ that $\W(X^{**},X)=\F_0(B^{**},B)$. By Theorem~\ref{t:B}, for every $n$ the pair $(nB^{**},nB)$ is strongly $\W(X^{**},X)$-universal and $nB^{**}$ is a $Z$-set in $(n+1)B^{**}$. Thus it is legal to apply Proposition~\ref{p3} to conclude that the pair $(X^{**}_c,X_c)$ is $\W(X^{**},X)$-universal.
\end{proof}

For a normed space $X$ denote by $(X,s)$ the space $X$ endowed with the strongest topology coinciding with the weak topology on bounded subsets of $X$. It is clear that $(X,s)$ is nothing else but the direct limit $\varinjlim nB$ of the tower $B\subset 2B\subset\cdots$, where $B$ is the weak unit ball of $X$. It is clear that the topology $s$ is $(X,s)$ is sequentially weak. Unlike to the spaces $(X,\weak)$ and $(X,c)$, for every normed space $X$ with separable dual, the space $(X,s)$ is sequential.

\begin{proposition}\label{p7} For every infinite-dimensional normed space $X$ with separable dual the space $(X,s)$ is $\W(X)$-injective,
\end{proposition}

\begin{proof} It follows from the definition of $(X,s)$ that $(X,s)=\varinjlim nB$. For every $n$ the space $nB$ can be considered as a dense convex subset of the $*$-weak ball $nB^{**}$ which is known to be a metrizable compact convex subset of the second dual space $X^{**}$ endowed with the $*$-weak topology, see Proposition 62 and Theorem 64 in \cite{19}.

Clearly, each ball $nB$ belongs to the class $\W(X)$. Next, $nB$, being a metrizable convex subset of a locally convex space, is a metrizable absolute retract. By Theorem~\ref{t:B}(iii), $nB$ is a $Z$-set in $(n+1)B$ for every $n$. Since $\W(X)=\F_0(B)$, Theorem~\ref{t:B}(i) implies that each space $nB$ is stronghly $\W(X)$-universal. Then by Proposition~\ref{p4}, the space $X_s$ is $\W(X)$-injective. 
\end{proof}
\smallskip

{\bf Proof of Theorem 1.} Let $X,Y$ be normed spaces with separable duals. Without loss of generality, the spaces $X,Y$ are infinite-dimensional.

To prove the ``if part'' of Theorem~\ref{t1}, suppose that $\W(X)=\W(Y)$. Since the identity map $(X,s)\to(X,\weak)$ and $(Y,s)\to (Y,\weak)$ are sequential homeomorphisms, to prove that the spaces $(X,\weak)$ and $(Y,\weak)$ are sequentially homeomorphic, it suffices to verify the topological equivalence of the spaces $(X,s)$ and $(Y,s)$. This easily follows from the Uniqueness Theorem~\ref{t:A} and Proposition~\ref{p7}. 

Next, assume that the spaces $(X,\weak)$ and $(Y,\weak)$ are sequentially homeomorphic. Then the spaces $(X,s)$ and $(Y,s)$ are sequentially homeomorphic too. Since these spaces are sequential, they are homeomorphic.

To show that $\W(X)=\W(Y)$, fix any space $A\in\W(X)$. Since each bounded subset of $(X,\weak)$ is metrizable and separable, see Proposition 62 of \cite{19}, we get that $A$ is a separable metrizable space. We may assume that $A$ is a closed subspace of $(X,s)$. Since the space $(Y,s)$ is homeomorphic to $(X,s)$, $A$ is homeomorphic to a closed subset $A'$ of the space $(Y,s)$. To show that $A\in\W(Y)$ it suffices to verify that $A'\in\W(Y)$.

We consider separately two cases:
\smallskip

1) The space $Y$ is reflexive. Then each space from the class $\W(Y)=\W(X)$ is compact. Consequently, the spaces $A$ and $A'$ are compact too and thus $A'\subset nB_Y$ for some $n$, where $B_Y$ stands for the wek unit ball of the space $Y$. Consequently, $A'\in\F_0(B_Y)=\W(Y)$. 
\smallskip

2) The space $Y$ is not reflexive. In this case $B_Y\ne B_Y^{**}$, where $B_Y^{**}$ stands for the unit ball of $Y^{**}$ endowed with the $*$-weak topology. By Theorem~\ref{t:B}(ii), the class $\F_0(B_Y)=\W(Y)$ is local. Hence, to show that $A'\in\W(Y)$, it suffices to verify that each point $a\in A"$ has a neighborhood $U\in\W(Y)$. Since the space $A'$ is first-countable, the point $a\in A"$ has a closed neighborhood $U\subset A'$ lying in $nB_Y$ for some $n\in\IN$. Then $U\in\F_0(B_Y)=\W(Y)$ and consequently, $A'\in W(Y)$.

Thus $\W(X)\subset \W(Y)$. By analogy we may prove that $\W(Y)\subset\W(X)$.
\smallskip

{\bf Proof of Theorem~\ref{t3}.} Let $X,Y$ be normed spaces with separable duals. If $\W(X^{**},X)=\W(Y^{**},Y)$, then the Uniqueness Theorem~\ref{t:A} and Proposition~\ref{p6} imply that the pairs $(X^{**}_c,X_c)$ and $(Y^{**}_c,Y_c)$ are homeomorphic.

To see that the homeomorphness of the pairs   $(X^{**}_c,X_c)$ and $(Y^{**}_c,Y_c)$ implies the equality $\W(X^{**},X)=\W(Y^{**},Y)$, observe that for a normed space $Z$ the class $\W(Z^{**},Z)$ coincides with the class of pairs homeomorphic to the pairs $(K,K\cap Z)$ where $K$ is a compact subset of $Z^{**}_c$ (this follows from the fact that the dual topology of compact convergence on $Z^{**}$ is the strongest topology inducing the $*$-weak topology on each bounded subset of $Z^{**}$, see \cite[IV.6.3]{25}).
\smallskip

{\bf Proof of Theorem~\ref{t4}.} Let $X$ be a normed space with separable dual and suppose that the class $\W(X^{**},X)$ is $[0,1]$-stable. Let $B$ denote the weak unit ball of $X$ and $B^{**}$ the $*$-weak unit ball of the second dual space $X^{**}$. According to the Uniqueness Theorem~\ref{t:A} and Proposition~\ref{p6}, to prove the homeomorphness of the pairs $(X^{**}_c,X_c)$ and $(B^{**}\times\IR^\infty,B\times \IR^\infty)$, it suffices to verify that the latter pair is $\W(X^{**},X)$-injective, But this follows from Proposition~\ref{p5}, Theorem~\ref{t:B}(i) and the evident equality $\W(X^{**},X)=\F_0(B^{**},B)$.
\smallskip

{\bf Proof of Theorem~\ref{t2}.} If for Banach spaces $X,Y$ with separable duals the spaces $(X,c)$ and $(Y,c)$ are homeomorphic, then the spaces $(X,\weak)$ and $(X,\weak)$ are sequentially homeomorphic and by Theorem~\ref{t1}, $\W(X)=\W(Y)$.

Next, assume that $\W(X)=\W(Y)$ for Banach spaces  $X,Y$ with separable duals. The spaces $X,Y$ are separable (see Proposition 51 in \cite{19}) and thus admit equivalent Kadec norms, see Theorem 111 in \cite{19}. We recall that a norm of a Banach space is {\em Kadec} if the weak and norm topologies coincide on the unit sphere. So, without loss of generality we may assume that the norms of the spaces $X,Y$ are Kadec. According to Theorem 1.14 in \cite{7}, the equality $\W(X)=\W(Y)$ implies the homeomorphness of the weak unit balls $B_X$ and $B_Y$ of the spaces $X,Y$. This implies the homeomorphness of the products $B_X\times\IR^\infty$ and $B_Y\times\IR^\infty$. Now Corollary~\ref{c2} implies the homeomorphness of the spaces $(X,c)$ and $(Y,c)$, provided the Banach spaces $X,Y$ are isomorphic to their hyperplanes.
\smallskip

{\bf Proof of Proposition~\ref{p1}.} Taking into account that the topology of compact convergence is stronger than the weak topology and coincides with it on bounded sets, we conclude that the identity map $(X,c)\to(X,\weak)$ is a sequential homeomorphism.

To show that for any infinite-dimensional normed space $X$ with separable dual the spaces $(X,c)$ and $(X,\weak)$ are not homeomorphic, we shall use suitable properties of stratifiable spaces, see \cite{18}. It is known \cite{12} that subspaces of spaces carrying the direct limit topology with respect to a tower of metrizable compacta are stratifiable. In particular, the space $X_c$ is stratifable. On the other hand, by Theorem 6a in \cite{17}, the weak topology of any infinite-dimensional normed space is not stratifiable.


\newpage 

\begin{thebibliography}{99}

\bibitem{1}  R.~Arens, {\em Duality in linear spaces}, Duke Math. J. {\bf 14} (1947), 787--794. 

\bibitem{2} J.~Baars, H.~Gladdines, J.~van Mill, {\em Absorbing systems in infinite-dimensional manifolds}, Topology Appl. {\bf 50}:2 (1993), 147--182.

\bibitem{3} S.~Banach, {\em Th\'eorie des opr\'erations lin\'eaires}, PWN, Warsaw, 1932.

\bibitem{4} T.~Banakh, {\em Toward a topological classification of convex sets in infinite-dimensional Fréchet spaces}, Topology Appl. {\bf 111}:3 (2001),  241--263. 

\bibitem{5} T.~Banakh, {\em On hyperspaces and homeomorphism groups homeomorphic to products of absorbing sets and $\mathbb R^\infty$}, Tsukuba J. Math. {\bf 23}:3 (1999),  495--504.

\bibitem{6} T.~Banakh, {\em On linear topological spaces \textup{(}linearly\textup{)} homeomorphic to ${\mathbb R}^\infty$}, Mat. Stud. {\bf 9}:1 (1998), 99--101.

\bibitem{7} T.~Banakh, {\em On topological classification of weak unit balls in Banach spaces}, Dissert. Math. {\bf 387} (2000), 7--35.

\bibitem{8} T.~Banakh, R.~Cauty, {\em Interplay between strongly universal spaces and pairs}, Dissert. Math. {\bf 386} (2000), 38 pp.

\bibitem{9} T.~Banakh, T.~Radul, M.~Zarichnyi, {\em Absorbing sets in infinite-dimensional manifolds}, VNTL Publ., Lʹviv, 1996. 232 pp.

\bibitem{10} C.~Bessaga, A.~Pe\l czy\'nski, {\em Selected topics in infinite-dimensional topology}, PWN, Warsaw, 1975. 353 pp.

\bibitem{11} M.~Bestvina, J.~Mogilski, {\em Characterizing certain incomplete infinite-dimensional absolute retracts}, Michigan Math. J. {\bf 33}:3 (1986), 291--313. 

\bibitem{12} C.~Borges, {\em On stratifiable spaces}, Pacific J. Math. {\bf 17} (1966), 1--16. 

\bibitem{13} R.~Cauty, {\em The strong universality property and its applications},  Proc. Steklov Inst. Math. {\bf 212}:1 (1996), 89--114.

\bibitem{14} R.~Edwards, {\em Functional analysis. Theory and applications}, Holt, Rinehart and Winston, New York-Toronto-London (1965), {\rm xiii}+781 pp. 

\bibitem{15} R.~Engelking, {\em General topology}, Heldermann Verlag, Berlin, (1989), {\rm viii}+529 pp.

\bibitem{16} M.~Fr\'echet, {\em Les espaces abstraits}, Hermann, Paris, 1928.

\bibitem{17} P.~Gartside, {\em Nonstratifiability of topological vector spaces}, Topology Appl. {\bf 86}:2 (1998), 133--140.

\bibitem{18} G.~Gruenhage, {\em Generalized metric spaces}, in: Handbook of set-theoretic topology,  North-Holland, Amsterdam, (1984), 423--501.

\bibitem{19} P.~Habala, P.~H\'ajek, V.~Zizler, {\em Introduction to Banach spaces}, Matfyzpress, Praha, 1996.

\bibitem{20} R.~Heisey, {\em Manifolds modelled on $\IR\sp{\infty }$ or bounded weak-$^*$ topologies}, Trans. Amer. Math. Soc. {\bf 206} (1975), 295--312.

\bibitem{21} M.I.~Kadec, {\em A proof of the topological equivalence of all separable infinite-dimensional Banach spaces}, Funkcional. Anal. i Prilo\v zen. {\bf 1} (1967) 61--70.

\bibitem{22} P.~Mankiewicz, {\em On topological, Lipschitz, and uniform classification of LF-spaces}, Studia Math. {\bf 52} (1974), 109--142.

\bibitem{23} E.~Pentsak, {\em On manifolds modeled on direct limits of $\mathcal C$-universal ANR's}, Mat. Stud. {\bf 5} (1995), 107--116.

\bibitem{24} K.~Sakai, {\em On ${\mathbb  R}\sp{\infty }$-manifolds and $Q\sp{\infty }$-manifolds}, Topology Appl. {\bf 18}:1 (1984), 69--79. 

\bibitem{25} H.~Schaefer, {\em Topological vector spaces}, The Macmillan Co., New York, (1966), {\rm ix}+294 pp. 

\bibitem{26} H.~Toru\'nczyk, {\em Characterizing Hilbert space topology}, Fund. Math. {\bf 111}:3 (1981), 247--262. 
\end{thebibliography}
\end{document}